\documentclass[11pt,a4paper]{article}
\usepackage{mathrsfs, mathtools, amsthm, amssymb, bm, thm-restate}
\usepackage{graphicx, xcolor, tikz}
\usetikzlibrary{calc, shapes}
\usepackage{caption}
\usepackage[labelformat=simple]{subcaption}

\makeatletter
\def\th@plain{%
  \upshape 
}
\makeatother

\makeatletter
\renewenvironment{proof}[1][\proofname]{\par
  \pushQED{\qed}%
  \normalfont \topsep6\p@\@plus6\p@\relax
  \trivlist
  \item[\hskip\labelsep
        \bfseries
    #1\@addpunct{.}]\ignorespaces
}{%
  \popQED\endtrivlist\@endpefalse
}
\makeatother

\newtheorem{theorem}{Theorem}[section]
\newtheorem{lemma}[theorem]{Lemma}
\newtheorem{corollary}[theorem]{Corollary}
\newtheorem{case}{Case}

\theoremstyle{definition}

\newtheorem{proposition}{Proposition}

\usepackage[left=25mm, top=25mm, bottom=25mm, right=25mm]{geometry}
\usepackage[T1]{fontenc}
\usepackage[inline]{enumitem}
\usepackage[square, numbers, sort&compress]{natbib}

\usepackage[
  bookmarks=true,%
  bookmarksnumbered=true, 
  bookmarksopen=true, 
  plainpages=false,%
  colorlinks=true, 
  linkcolor=black, 
  citecolor=black,%
  anchorcolor=green,
  urlcolor=blue,
  hyperindex=true]{hyperref}

\newcommand{\etal}{et~al.\ }

\usepackage{marvosym,wasysym}

\def\int(#1){\mathrm{int}(#1)}
\def\ext(#1){\mathrm{ext}(#1)}
\def\Int(#1){\mathrm{Int}(#1)}
\def\Ext(#1){\mathrm{Ext}(#1)}

\renewcommand{\emph}{\textbf}

\newcommand{\drawnode}[1]{\node[circle, inner sep = 1.2, fill, draw] () at (#1) {};}
\title{Planar graphs without $5^{-}$-cycles at distance less than $3$ are $(\mathcal{I}, \mathcal{F})$-colorable}
\author{Zhen He\footnote{School of Mathematics and Statistics, Henan University, Kaifeng, 475004, P. R. China} \and Tao Wang\footnote{Center for Applied Mathematics, Henan University, Kaifeng, 475004, P. R. China. {\tt Corresponding
author: wangtao@henu.edu.cn; https://orcid.org/0000-0001-9732-1617} } \and Xiaojing Yang\footnote{School of Mathematics and Statistics, Henan University, Kaifeng, 475004, P. R. China}}
\begin{document}
\date{}
\maketitle

\begin{abstract}
A graph is $(\mathcal{I}, \mathcal{F})$-colorable if its vertex set can be partitioned into two subsets, one of which is an independent set, and the other induces a forest. In this paper, we prove that every planar graph without $5^{-}$-cycles at distance less than $3$ is $(\mathcal{I}, \mathcal{F})$-colorable. 
\end{abstract}

\section{Introduction}
\label{sec:1}

An \emph{$(\mathcal{I}, \mathcal{F})$-coloring}, or \emph{\textit{IF}-coloring} for short, of a graph is a vertex coloring $\phi$ using the colors $I$ and $F$ such that $\mathcal{I} \coloneqq \{v \mid \phi(v) = I\}$ forms an independent set and $\mathcal{F} \coloneqq \{v \mid \phi(v) = F\}$ induces a forest. A graph is \emph{$(\mathcal{I}, \mathcal{F})$-partitionable}, or \emph{$(\mathcal{I}, \mathcal{F})$-colorable}, if the vertex set can be partitioned into two parts where one forms an independent set, and the other induces a forest. Borodin and Glebov \cite{MR1918259} proved that every planar graph with girth at least $5$ is $(\mathcal{I}, \mathcal{F})$-colorable. The result was further extended by Kawarabayashi and Thomassen \cite{MR2518200}. Furthermore, if we require the maximum degree of $\mathcal{F}$ is bounded by an integer $d$, we can define the concept of $(\mathcal{I}, \mathcal{F}_{d})$-coloring. Dross, Montassier and Pinlou \cite{MR3785024} considered the $(\mathcal{I}, \mathcal{F}_{d})$-colorability of graphs with bounded maximum average degree. Chen \etal \cite{MR3759461} strengthened the result by showing that every graph with maximum average degree at most $2 + \frac{d}{d+1}$ admits an $(\mathcal{I}, \mathcal{F}_{d})$-coloring. Cranston and Yancey \cite{MR4134028} considered the $(\mathcal{I}, \mathcal{F}_{d})$-colorability of graphs with bounded maximum average degree and some forbidden configurations. 

Let $C$ be a cycle in $G$. A chord of $C$ is a \emph{$(l_{1}, l_{2})$-chord} if the chord split $C$ into two cycles of lengths $l_{1}$ and $l_{2}$. If there exists a vertex $v \notin V(C)$ with three neighbors $v_{1}, v_{2}$ and $v_{3}$ on $C$, then we call $G[\{vv_{1}, vv_{2}, vv_{3}\}]$ a claw of $C$, see \autoref{claw}. A cycle and one of its claws form three cycles of lengths $c_{1}, c_{2}$ and $c_{3}$, we say that the claw is a \emph{$(c_{1}, c_{2}, c_{3})$-claw}. A $k$-cycle is a cycle of length $k$. A $9$-cycle is defined as \emph{special} if it has a $(3, 8)$-chord or a $(5, 5, 5)$-claw. Let $\mathcal{G}$ denote the family of connected planar graphs without $4$-, $6$-cycles, and special $9$-cycles. Recently, Kang, Lu and Jin \cite{arXiv:2303.04648} proved the following result regarding $(\mathcal{I}, \mathcal{F})$-coloring.
\begin{theorem}[Kang, Lu and Jin \cite{arXiv:2303.04648}]\label{IF-Kang}
Every graph in $\mathcal{G}$ is $(\mathcal{I}, \mathcal{F})$-colorable.
\end{theorem}

The following two corollaries are immediate consequence of \autoref{IF-Kang}.
\begin{corollary}[Liu and Yu \cite{MR4115511}]\label{Liu-Yu}
Every planar graph without $4$-, $6$-, and $8$-cycles is $(\mathcal{I}, \mathcal{F})$-colorable.
\end{corollary}

\begin{corollary}[Kang, Lu and Jin \cite{arXiv:2303.04648}]
Every planar graph without $4$-, $6$-, and $9$-cycles is $(\mathcal{I}, \mathcal{F})$-colorable.
\end{corollary}

\begin{figure}%
\centering
\def\s{1.5}
\subcaptionbox{\label{claw}A $(c_{1}, c_{2}, c_{3})$-claw.}[0.4\linewidth]{\begin{tikzpicture}
\draw  circle (\s);
\draw[line width = 1pt] (0, 0)--(90:\s);
\draw[line width = 1pt] (0, 0)--(210:\s);
\draw[line width = 1pt] (0, 0)--(330:\s);
\drawnode{0, 0};
\drawnode{90:\s};
\drawnode{210:\s};
\drawnode{330:\s};
\node at (30:0.5*\s) {$c_{1}$};
\node at (150:0.5*\s) {$c_{2}$};
\node at (270:0.5*\s) {$c_{3}$};

\end{tikzpicture}}
\subcaptionbox{\label{triclaw}A $(c_{1}, c_{2}, c_{3})$-claw.}[0.4\linewidth]{\begin{tikzpicture}
\draw circle (\s);
\draw[line width = 1pt] (90:0.4*\s)--(90:\s);
\draw[line width = 1pt] (210:0.4*\s)--(210:\s);
\draw[line width = 1pt] (330:0.4*\s)--(330:\s);
\draw[line width = 1pt] (90:0.4*\s)--(210:0.4*\s)--(330:0.4*\s)--cycle;
\drawnode{90:0.4*\s};
\drawnode{210:0.4*\s};
\drawnode{330:0.4*\s};
\drawnode{90:\s};
\drawnode{210:\s};
\drawnode{330:\s};
\node at (30:0.6*\s) {$c_{1}$};
\node at (150:0.6*\s) {$c_{2}$};
\node at (270:0.6*\s) {$c_{3}$};
\node at (0, 0) {$3$};
\end{tikzpicture}}
\caption{A $(c_{1}, c_{2}, c_{3})$-claw and a $(c_{1}, c_{2}, c_{3})$-triclaw.}
\end{figure}
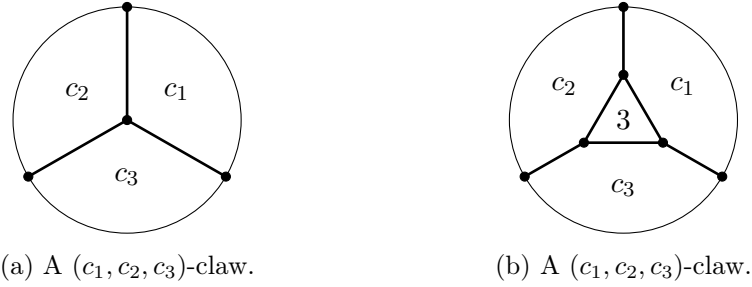

Two cycles are \emph{adjacent} if they share at least one common edge (potentially sharing more than two common vertices). Similarly, two cycles are \emph{normally adjacent} if they share exactly one common edge and exactly two common vertices. Lu \etal \cite{MR4422988} studied planar graphs without certain normally adjacent cycles. A cycle is a $l^{-}$-cycle if its length is at most $k$. Let $\mathscr{G}$ be the class of planar graphs without triangles normally adjacent to $8^{-}$-cycles, without $4$-cycles normally adjacent to $6^{-}$-cycles, and without normally adjacent $5$-cycles. Lu \etal strengthened \autoref{Liu-Yu} to the graph class $\mathscr{G}$. 

\begin{theorem}[Lu \etal \cite{MR4422988}]
Every planar graph in $\mathscr{G}$ is $(\mathcal{I}, \mathcal{F})$-colorable.
\end{theorem}

Let $d^{*}$ denote the minimum distance between two $5^{-}$-cycles. Yin and Yu considered planar graphs with restrictions on the distance of triangles, and proved the following result on DP-$3$-coloring. For the definition of DP-coloring, we refer the reader to \cite{MR3758240,MR3686937}.  
\begin{theorem}[Yin and Yu \cite{MR3954054}]
Every planar graph with $d^{*} \geq 3$ and without 4-, 5-cycles is DP-$3$-colorable. 
\end{theorem}

Zhao \cite{Zhao_2020} improved the above theorem to the following. 

\begin{theorem}[Zhao \cite{Zhao_2020}]
Every planar graph with $d^{*} \geq 3$ is DP-$3$-colorable. 
\end{theorem}

In this paper, we demonstrate that the same family of graphs as in \cite{Zhao_2020} is $(\mathcal{I}, \mathcal{F})$-colorable.
\begin{theorem}\label{Colorable}
Every planar graph with $d^{*} \geq 3$ is $(\mathcal{I}, \mathcal{F})$-colorable.
\end{theorem}

An \emph{$F$-path} (respectively, \emph{$F$-cycle}) is a path (respectively, cycle) where all vertices are colored with $F$. Given a graph $G$ and a cycle $Q$ in $G$, a \emph{splitting $F$-path} of $Q$ is an $F$-path whose two endpoints are on $Q$ and all other vertices (at least one) are not on $Q$.  A \emph{super \textit{IF}-coloring} of $(G, Q)$ is an \textit{IF}-coloring with the additional property that no splitting $F$-paths exist.

\begin{figure}%
\centering
\def\s{1.5}
\subcaptionbox{Type I. \label{FB1}}[0.4\linewidth]{\begin{tikzpicture}
\draw [line width = 1.5pt] circle (\s);
\draw (0, 0)--(90:\s);
\draw (0, 0)--(250:\s);
\draw (0, 0)--(290:\s);
\drawnode{0, 0};
\node at (10:0.5*\s) {$6$};
\node at (170:0.5*\s) {$6$};
\node at (270:0.5*\s) {$3$};

\foreach \ang in {1, ..., 9}{
\drawnode{\ang*40+50:\s};
\node at (\ang*40+50:1.2*\s) {$v_{\scriptscriptstyle\ang}$};}
\end{tikzpicture}}
\subcaptionbox{Type II. \label{FB2}}[0.4\linewidth]{\begin{tikzpicture}
\draw [line width = 1.5pt] circle (\s);
\draw (90:0.4*\s)--(90:\s);
\draw (210:0.4*\s)--(210:\s);
\draw (330:0.4*\s)--(330:\s);
\draw (90:0.4*\s)--(210:0.4*\s)--(330:0.4*\s)--cycle;
\drawnode{90:0.4*\s};
\drawnode{210:0.4*\s};
\drawnode{330:0.4*\s};
\node at (30:0.6*\s) {$6$};
\node at (150:0.6*\s) {$6$};
\node at (270:0.6*\s) {$6$};
\node at (0, 0) {$3$};

\foreach \ang in {1, ..., 9}{
\drawnode{\ang*40+50:\s};
\node at (\ang*40+50:1.2*\s) {$v_{\scriptscriptstyle\ang}$};}
\end{tikzpicture}}
\caption{Two types of bad $9$-cycles.}
\label{S9}
\end{figure}
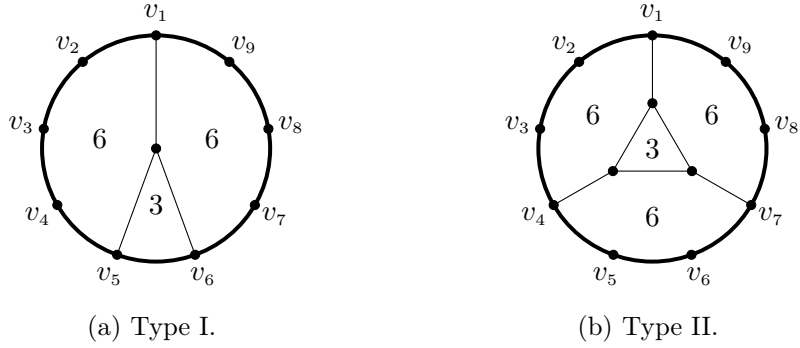

For a cycle $C$, if there exists three pairwise adjacent vertices $x, y, z \notin V(C)$ with three neighbors $x', y'$ and $z'$ on $C$, then we call $G[\{xy, yz, xz, xx', yy', zz'\}]$ a \emph{triclaw} of $C$, see \autoref{triclaw}. A \emph{bad cycle} is a $9$-cycle that has a claw or a triclaw, \emph{good} otherwise. Observe that every $8^{-}$-cycle is good. In \autoref{S9}, the two triangles are called \emph{associated triangles} of the bad cycles, where $v_{4}$ and $v_{7}$ are called \emph{bad vertices} of associated triangles. Note that the two bad vertices are at distance $3$ along the bad cycle. Since $d^{*} \geq 3$, no vertex on the bad cycle of type II is contained in a $5^{-}$-cycle, and no vertex  $\neq v_{5}, v_{6}$ on the bad cycle of type I is contained in a $5^{-}$-cycle. 

\begin{theorem}\label{SuperColorable}
Assume $G$ is a planar graph in which the $5^{-}$-cycles are at least $3$ apart. If $C_{0}$ is a good $9^{-}$-cycle of $G$, then every $(\mathcal{I}, \mathcal{F})$-coloring $\phi_{0}$ of $G[V(C_{0})]$ can be extended to a super \textit{IF}-coloring of $(G, C_{0})$.
\end{theorem}
\begin{proof}[Proof of \autoref{Colorable} from \autoref{SuperColorable}]
Assume $G$ is a planar graph as described in \autoref{Colorable}. Let $H \coloneqq G$ if $G$ has a $5^{-}$-cycle $C_{0}$. Otherwise, let $H$ be obtained from $G$ by adding a triangle $xyz$, namely $C_{0}$, such that $x$ is a cut-vertex of $H$. Clearly, in both cases $C_{0}$ is an induced cycle, so it has an $(\mathcal{I}, \mathcal{F})$-coloring $\phi_{0}$. By \autoref{SuperColorable}, $(H, C_{0})$ has a super \textit{IF}-coloring $\phi$ such that its restriction on $C_{0}$ is $\phi_{0}$. Observe that the restriction of $\phi$ on $G$ is an $(\mathcal{I}, \mathcal{F})$-coloring of $G$. This completes the proof of \autoref{Colorable}.
\end{proof}

Let $H$ be a subgraph of a graph $G$. Given an $(\mathcal{I}, \mathcal{F})$-coloring of $H$, to \emph{nicely color} an uncolored vertex $v$ means assigning the color $I$ to $v$ if $v$ has no neighbors of color $I$, and assigning $F$ otherwise. For an uncolored path $P = x_{1}x_{2}\dots x_{m}$, to \emph{nicely color $P$} means to nicely color $x_{1}, x_{2}, \dots, x_{m}$ sequentially. 

It is easy to obtain the following proposition. 
\begin{proposition}\label{prop}
Assume $G$ is a graph and $S = u_{1}, u_{2}, \dots, u_{t}$ is a sequence of vertices in $G$.  Let $\phi$ be an \textit{IF}-coloring of a subgraph $H$ of $G - V(S)$. Fix a cycle $Q$ in $H$. If each $u_{i}$ has at most two neighbors in $V(H) \cup \{u_{1}, \dots, u_{i-1}\}$, then the nicely coloring of $S$ would not create new $F$-cycles or new splitting $F$-paths of $Q$. Moreover, if $u_{1}u_{2}\dots u_{t}$ is a path, then there is no $F$-path with the pattern $u_{i}u_{i+1}v_{i+1}$, where $u_{i}$ and $v_{i+1}$ are two colored neighbor of $u_{i+1}$ when we try to color $u_{i+1}$. 
\end{proposition}

To end this section, we introduce some notations. For a plane graph $G$, we use $V(G), E(G)$, and $F(G)$ to denote the set of vertices, edges, and faces, respectively. When $C$ is a cycle in $G$, we use $\int(C)$ and $\ext(C)$ to denote the set of vertices located inside and outside of $C$, respectively. A $k$-vertex (resp. $k^{+}$-vertex, and $k^{-}$-vertex) is a vertex of degree $k$ (resp. at least $k$, and at most $k$). Similarly, a $k$-face (resp. $k^{+}$-face, and $k^{-}$-face) is a face of degree $k$ (resp. at least $k$, and at most $k$). An $(l_{1}, l_{2}, \dots, l_{k})$-face is a $k$-face $[v_{1}v_{2}\dots v_{k}]$ with $d(v_{i}) = l_{i}$ for all $1 \leq i \leq k$. An $(l_{1}, l_{2})$-edge is an edge with two ends having degree $l_{1}$ and $l_{2}$. 
 
\section{Proof of \autoref{SuperColorable}}
\label{sec:2}

We will prove \autoref{SuperColorable} by contradiction. Let $(G, C_{0}, \phi_{0})$ be a counterexample to the statement. Furthermore, we select this counterexample so that the quantity $|V(G)| + |E(G)|$ is as small as possible. In this instance, $\phi_{0}$ cannot be extended to a super \textit{IF}-coloring of $(G, C_{0})$. We assume that $G$ has been embedded on the plane, and by minimality, $C_{0}$ has no chords.

A cycle is a \emph{separating cycle} if neither its interior nor exterior is empty. 
\begin{lemma}\label{SEP}
There are no separating good $9^{-}$-cycles. Furthermore, we can assume that $C_{0}$ bounds the outer face $D$ of $G$. 
\end{lemma}
\begin{proof}
First, we prove that $C_{0}$ is not a separating cycle. Suppose, for contradiction, that $C_{0}$ is a separating cycle. This implies that $\int(C_{0}) \neq \emptyset$ and $\ext(C_{0}) \neq \emptyset$. By minimality, we can extend $\phi_{0}$ to a super \textit{IF}-coloring $\phi_{1}$ of $(G - \int(C_{0}), C_{0})$ and a super \textit{IF}-coloring $\phi_{2}$ of $(G - \ext(C_{0}), C_{0})$. Combining these two colorings, $\phi_{1}$ and $\phi_{2}$, we obtain a super \textit{IF}-coloring of $(G, C_{0})$, a contradiction. Since $C_{0}$ has no chords and is an induced cycle, we can assume that it bounds the outer face of $G$.  

Next, we show that there are no other separating good $9^{-}$-cycles. Let $C$ be a separating good $9^{-}$-cycle of $G$. By minimality, we can extend $\phi_{0}$ to a super \textit{IF}-coloring of $(G - \int(C), C_{0})$. Furthermore, the coloring of $G[V(C)]$ can be further extended to a super \textit{IF}-coloring of $(G - \ext(C), C)$. Combining these two colorings, we obtain a super \textit{IF}-coloring of $(G, C_{0})$, a contradiction.  
\end{proof}

By \autoref{SEP}, every claw or triclaw in a bad cycle $C'$ partitions the region bounded $C'$ into faces. We define a \emph{boundary vertex} as a vertex on the outer face. A vertex is \emph{internal} if it is not a boundary vertex, and a face is \emph{internal} if it contains no boundary vertices. 

\begin{lemma}\label{mindeg}
Each internal vertex has degree at least $3$. 
\end{lemma}
\begin{proof}
Let $v$ be an internal vertex of degree at most two. By minimality, we can extend $\phi_{0}$ to a super \textit{IF}-coloring of $G - v$. Nicely coloring $v$ can be incorporated into the extended super \textit{IF}-coloring of $G -v$, resulting in a super \textit{IF}-coloring of the entire graph $G$. This contradicts our initial assumption that $(G, C_{0}, \phi_{0})$ is a counterexample.
\end{proof}

It is easy to obtain the following result. 
\begin{lemma}\label{normally}
Let $f_{1}$ be a $6$-face and $f_{2}$ be a $5^{-}$-face. If $f_{1}$ and $f_{2}$ have a common edge $xy$, then they are normally adjacent, and $G[V(f_{1}) \cup V(f_{2})]$ is a cycle with a unique chord. Moreover, no vertex in $V(f_{1}) \setminus \{x, y\}$ is contained in a $5^{-}$-cycle. 
\end{lemma}

A \emph{barrier} is defined as either an $F$-cycle or a splitting $F$-path. 

\begin{lemma}\label{34Chord}
An internal $(3, 3, 3, 3, 3, 4)$-face cannot share a $(3, 4)$-edge with an internal $5^{-}$-face, in which all except the $4$-vertex are $3$-vertices.  
\end{lemma}
\begin{proof}
Let $f = [v_{1}v_{2}\dots v_{5}v_{t}]$ be an internal $(3, 3, 3, 3, 3, 4)$-face, and let $[v_{5}v_{6}\dots v_{t}]$ be an internal $(3, 3, \dots, 4)$-face where $d(v_{5}) = 4$. By \autoref{normally}, the subgraph induced by $v_{1}, v_{2}, \dots, v_{t}$ is a cycle with a unique chord $v_{5}v_{t}$. For $1 \leq i \leq t - 1$, let $u_{i}$ be the remaining neighbor of $v_{i}$.  By the minimality of $G$, we can extend $\phi_{0}$ to a super \textit{IF}-coloring of $G - \{v_{1}, v_{2}, \dots, v_{t}\}$. First, we nicely color the sequence $v_{1}, v_{2}, \dots, v_{t-1}$, and denoting the resulting coloring of $G - v_{t}$ by $\phi$. Note that $\phi$ is a super \textit{IF}-coloring of $G - v_{t}$ by \autoref{prop}. If $v_{1}, v_{5}$ and $v_{t-1}$ are all colored with $F$, then we color $v_{t}$ with $I$. If at most one of $v_{1}, v_{5}$ and $v_{t-1}$ is colored with $F$, then we color $v_{t}$ with $F$. So we may assume that exactly one of $v_{1}, v_{5}$ and $v_{t}$ is colored with $I$ and the other two are colored with $F$. We then color $v_{t}$ with $F$, and denote the resulting coloring by $\phi'$. If $\phi'$ is a super \textit{IF}-coloring, then the proof is completed. Otherwise, every barrier under $\phi'$ contains the vertex $v_{t}$. Let $Q$ be an arbitrary barrier with respective to $\phi'$. In the following, we will recolor some vertices so that all the barriers are destroyed and no new barriers are created.   

\begin{case}
$\phi(v_{1}) = F$. 
\end{case}
Since $v_{1}$ is nicely colored, we have $\phi(u_{1}) = I$. Thus, the barrier $Q$ pass through $v_{t}v_{1}v_{2}$. By \autoref{prop}, there is no $F$-path with the pattern $v_{i}v_{i+1}u_{i+1}$ for $1 \leq i \leq t-2$. Therefore, $Q$ must lie entirely within the subgraph $G[v_{1}, v_{2}, \dots, v_{t}]$, which implies that $v_{1}v_{2}\dots v_{5}$ is an $F$-path. Recall that $v_{t}$ has a neighbor with color $I$, we have $\phi(v_{t-1}) = I$. Once again, the nicely coloring of $v_{1}v_{2}\dots v_{t-1}$ implies that $\phi(u_{1}) = \phi(u_{2}) = \dots = \phi(u_{5}) = I$ and $\phi(u_{t-1}) = \phi(v_{t-2}) = F$. Now, we revise $\phi'$ by assigning $F$ to $v_{t-1}$ and $I$ to $v_{t}$. Denote the resulting coloring $\psi$. If there is no barrier under $\psi$, then we are done. Hence, there is a barrier passing through $u_{r}v_{r}\dots v_{t-1}u_{t-1}$ under $\psi$. Consequently, $\psi(v_{r-1}) = I$ and $r - 1 \geq 6$. Since $[v_{5}v_{6}\dots v_{t}]$ is a $5^{-}$-face, we have $t = 9$ and $r = 7$. We further revise $\psi$ by assigning $F$ to $v_{6}$ and $I$ to $v_{7}$, obtaining a desired super \textit{IF}-coloring of $G$. 

\begin{case}
$\phi(v_{1}) = I$. 
\end{case}
Then $\phi(v_{5}) = \phi(v_{t-1}) = F$. Since $v_{1}v_{2}\dots v_{t-1}$ is nicely colored, we have $\phi(u_{1}) = F$, and either $\phi(v_{4}) = I$ or $\phi(u_{5}) = I$. 

\begin{itemize}
\item $\phi(u_{2}) = I$. We revise $\phi'$ by letting $\phi'(v_{1}) = F$, $\phi'(v_{t}) = I$. According to \autoref{prop}, it is easy to check that $v_{1}$ cannot be on a barrier after the revision. 

\item $\phi(u_{2}) = F$ and $\phi(u_{3}) = I$. We revise $\phi'$ by letting $\phi'(v_{1}) = F$, $\phi'(v_{t}) = I$ and $\phi(v_{2}) = I$. Note that $\phi(u_{2}) = \phi(v_{3}) = F$. Similarly, $v_{1}$ cannot be on a barrier after the revision. 

\item $\phi(u_{2}) = \phi(u_{3}) = F$ and $\phi(u_{4}) = I$. We revise $\phi'$ by letting $\phi'(v_{1}) = \phi'(v_{3}) = F$, $\phi'(v_{t}) = I$ and $\phi(v_{2}) = I$.  According to \autoref{prop}, it is easy to check that neither $v_{1}$ nor $v_{3}$ is on a barrier after the revision. 

\item $\phi(u_{2}) = \phi(u_{3}) = \phi(u_{4}) = F$. We revise $\phi'$ by letting $\phi'(v_{1}) = \phi'(v_{3}) = F$, $\phi'(v_{t}) = I$ and $\phi(v_{2}) =  \phi(v_{4}) = I$.  Similarly, neither $v_{1}$ nor $v_{3}$ is on a barrier after the revision.  
\end{itemize}

In both main cases, the conclusion is that we are able to find a valid super \textit{IF}-coloring of $G$, contradicting our initial assumption. Therefore, the lemma is proved.
\end{proof}

\begin{lemma}\label{Y}
Assume $[x_{1}x_{2}x_{3}\dots x_{k}]$ is an internal $5^{-}$-face where all vertices are $3$-vertices. Let $[y_{0}y_{1}x_{2}x_{3}y_{4}y_{5}\dots]$ be an adjacent $6^{+}$-face. If $y_{1}$ and $y_{4}$ are internal vertices, then $y_{1}$ or $y_{4}$ is a $4^{+}$-vertex. 
\end{lemma}
\begin{proof}
Assume, for a contradiction, that both $y_{1}$ and $y_{4}$ are internal $3$-vertices. Let $v_{1}$ be the third neighbor of $x_{1}$. Let $G_{0} = G - \{x_{1}, x_{2}, \dots, x_{k}, y_{1}, y_{4}\}$, and let $G_{1}$ be the graph obtained from $G_{0}$ by identifying $v_{1}$ and $y_{0}$ (note that $\{v_{1}, y_{0}\} \subset V(G_{0})$, and $v_{1} \neq y_{0}$). It is clear that $G_{1}$ is a plane graph. We claim that the distance of $v_{1}$ and $y_{0}$ in $G_{0}$ is at least $6$. Assuming that a path $P$ exists with length at most $5$ between $v_{1}$ and $y_{0}$ in $G_{0}$. Then $C = P \cup v_{1}x_{1}x_{2}y_{1}y_{0}$ is a cycle of length at most $9$. Since $d^{*} \geq 3$, the vertex $y_{1}$ is not on a $5^{-}$-cycle. This implies that the third neighbor of $y_{1}$ is not on $C$. Hence, $C$ is a separating $9^{-}$-cycle that separates $x_{3}$ and the third neighbor of $y_{1}$. Consequently, $C$ is a separating bad $9$-cycle by \autoref{SEP}. Since $x_{2}$ is on the bad cycle $C$ and also on the $5^{-}$-cycle $x_{1}x_{2}x_{3}\dots x_{k}x_{1}$, the bad cycle has a claw as depicted in \autoref{FB1}. This leads to $k = 3$ and $x_{1}x_{2}x_{3}$ is the associated triangle of the bad $9$-cycle $C$. Due to the feature of the bad cycle with a claw, the $6$-face incident with $x_{2}x_{3}$ should contain only one vertex not on $C$, but this contradicts that $x_{3}$ and $y_{4}$ are not on $C$. Hence, the claim is proved. 

By this claim, no new $5^{-}$-cycles are created in $G_{1}$, and the distance between two $5^{-}$-cycles is at least $3$ in $G_{1}$. Furthermore, $C_{0}$ still bounds the outer face of $G_{1}$, has no chords in $G_{1}$, and $\phi_{0}$ is an \textit{IF}-coloring of $C_{0}$. 

Let $G_{2}$ be the graph obtained from $G_{0}$ by identifying $x_{4}'$ and $y_{5}$, where $x_{4}'$ is the third neighbor of $x_{4}$ in $G$. By symmetry, $G_{2}$ satisfies similar properties as $G_{1}$: it is a plane graph, no new $5^{-}$-cycles are created in $G_{2}$, the distance between two $5^{-}$-cycles is at least $3$ in $G_{2}$, the cycle $C_{0}$ bounds the outer face of $G_{2}$, it has no chords in $G_{2}$, and $\phi_{0}$ is an \textit{IF}-coloring of $C_{0}$. Our next step is to show that $C_{0}$ is a good cycle of either $G_{1}$ or $G_{2}$. Assume that $C_{0}$ is a bad $9$-cycle of $G_{1}$. Then a new $6$-cycle is created. Since neither $v_{1}$ nor $y_{0}$ is on a $5^{-}$-cycle of $G$, the identified vertex is on $C_{0}$. Since $d^{*} \geq 3$, it implies that $v_{1}$ is on $C_{0}$ and $y_{0}$ is an internal vertex with a pendent triangle, whose bad vertices are $v_{4}$ and $v_{7}$, as illustrated in \autoref{BAD}. A straightforward check reveals that the identification of $x_{4}'$ and $y_{5}$ maintains $C_{0}$ as a good cycle in $G_{2}$. Otherwise, $x_{4}'$ would be on $C_{0}$ and $y_{5}$ would be an internal vertex with a pendent triangle, whose bad vertices lie in the segment $v_{1}v_{2}v_{3}v_{4}$, contradicting $d^{*} \geq 3$.

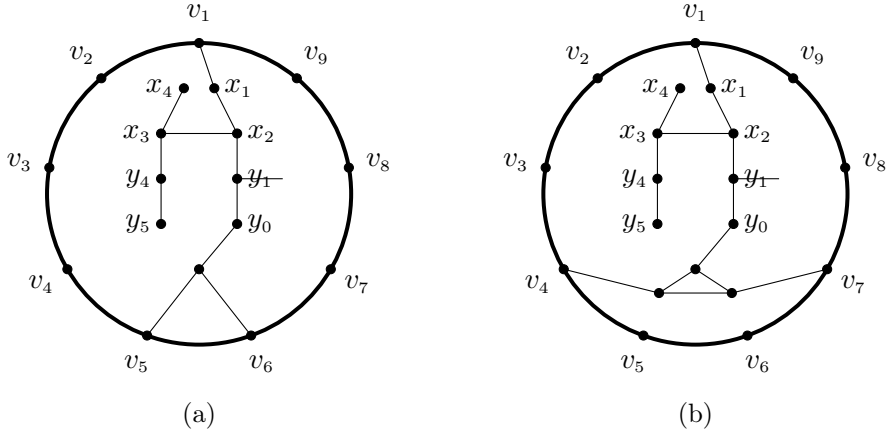
\begin{figure}%
\centering
\def\s{2}
\subcaptionbox{}[0.4\linewidth]{\begin{tikzpicture}
\draw [line width = 1.5pt] circle (\s);
\path 
    (0.1*\s, 0.7*\s) coordinate (x1) node[right]{$x_{\scriptscriptstyle 1}$}
    (-0.1*\s, 0.7*\s) coordinate (x4) node[left]{$x_{\scriptscriptstyle 4}$}
    (0.25*\s, 0.4*\s) coordinate (x2) node[right]{$x_{\scriptscriptstyle 2}$}
    (0.25*\s, 0.1*\s) coordinate (y1) node[right]{$y_{\scriptscriptstyle 1}$}
    (0.25*\s, -0.2*\s) coordinate (y0) node[right]{$y_{\scriptscriptstyle 0}$}
    (-0.25*\s, 0.4*\s) coordinate (x3) node[left]{$x_{\scriptscriptstyle 3}$}
    (-0.25*\s, 0.1*\s) coordinate (y4) node[left]{$y_{\scriptscriptstyle 4}$}
    (-0.25*\s, -0.2*\s) coordinate (y5) node[left]{$y_{\scriptscriptstyle 5}$};
\draw (90:\s)--(x1)--(x2)--(y1)--(y0)--(0, -0.5*\s);

\draw (0, -0.5*\s)--(250:\s);
\draw (0, -0.5*\s)--(290:\s);
\draw (y1)--($(y1)+(0.3*\s, 0)$);
\draw (x2)--(x3)--(y4)--(y5);
\draw (x3)--(x4);

\drawnode{0, -0.5*\s};
\foreach \v in {x1, x2, x3, x4, y0, y1, y4, y5}{\drawnode{\v};}
\foreach \ang in {1, ..., 9}{
\drawnode{\ang*40+50:\s};
\node at (\ang*40+50:1.2*\s) {$v_{\scriptscriptstyle\ang}$};}
\end{tikzpicture}}
\subcaptionbox{}[0.4\linewidth]{\begin{tikzpicture}
\draw [line width = 1.5pt] circle (\s);
\path 
    (0.1*\s, 0.7*\s) coordinate (x1) node[right]{$x_{\scriptscriptstyle 1}$}
    (-0.1*\s, 0.7*\s) coordinate (x4) node[left]{$x_{\scriptscriptstyle 4}$}
    (0.25*\s, 0.4*\s) coordinate (x2) node[right]{$x_{\scriptscriptstyle 2}$}
    (0.25*\s, 0.1*\s) coordinate (y1) node[right]{$y_{\scriptscriptstyle 1}$}
    (0.25*\s, -0.2*\s) coordinate (y0) node[right]{$y_{\scriptscriptstyle 0}$}
    (-0.25*\s, 0.4*\s) coordinate (x3) node[left]{$x_{\scriptscriptstyle 3}$}
    (-0.25*\s, 0.1*\s) coordinate (y4) node[left]{$y_{\scriptscriptstyle 4}$}
    (-0.25*\s, -0.2*\s) coordinate (y5) node[left]{$y_{\scriptscriptstyle 5}$};
\draw (90:\s)--(x1)--(x2)--(y1)--(y0)--(0, -0.5*\s);

\draw (0, -0.5*\s)--(250:0.7*\s)--(210:\s);
\draw (0, -0.5*\s)--(290:0.7*\s)--(330:\s);
\draw (250:0.7*\s)--(290:0.7*\s);
\draw (y1)--($(y1)+(0.3*\s, 0)$);
\draw (x2)--(x3)--(y4)--(y5);
\draw (x3)--(x4);

\drawnode{0, -0.5*\s};
\drawnode{250:0.7*\s};
\drawnode{290:0.7*\s};
\foreach \v in {x1, x2, x3, x4, y0, y1, y4, y5}{\drawnode{\v};}
\foreach \ang in {1, ..., 9}{
\drawnode{\ang*40+50:\s};
\node at (\ang*40+50:1.2*\s) {$v_{\scriptscriptstyle\ang}$};}
\end{tikzpicture}}
\caption{Local structure.}
\label{BAD}
\end{figure}

By symmetry, we may assume that $C_{0}$ is a good cycle in $G_{1}$. By the minimality of $G$, the precoloring $\phi_{0}$ can be extended to a super \textit{IF}-coloring of $(G_{1}, C_{0})$. This coloring corresponds to a super \textit{IF}-coloring of $G_{0}$ by assigning the same color to $v_{1}$ and $y_{0}$ as the new vertex in $G_{1}$. We use $\phi$ to denote this coloring of $G_{0}$.

We consider two cases.
\begin{itemize}
\item $\phi(v_{1}) = \phi(y_{0}) = I$. We color $x_{1}$ and $y_{1}$ with $F$, and then proceed to nicely color the sequence $x_{k}, x_{k-1}, \dots, x_{4}$, $y_{4}$ and $x_{3}$. By \autoref{prop}, the resulting coloring is a super \textit{IF}-coloring. If $x_{3}$ is colored with $F$, then we color $x_{2}$ with $I$; otherwise we color $x_{2}$ with $F$. It is easy to check that the final coloring is a super \textit{IF}-coloring of $G$. 

\item $\phi(v_{1}) = \phi(y_{0}) = F$. We color $y_{1}$ with a color different from the third neighbor of $y_{1}$. Observe that $y_{1}$ is not on a barrier. Next, we color $x_{1}$ with the same color as $y_{1}$. We then nicely color the sequence $x_{k}, x_{k-1}, \dots, x_{4}, y_{4}, x_{3}$. By \autoref{prop}, the resulting coloring has no barriers. If $x_{1}$ and $y_{1}$ are colored with $I$, then we color $x_{2}$ with $F$; otherwise, we color $x_{2}$ with a color different from $x_{3}$. The path $v_{1}x_{1}x_{2}y_{1}y_{0}$ is not on a barrier; otherwise, the identification of $v_{1}$ and $y_{0}$ would create a barrier in $G_{1}$. It is easy to check the final coloring is a super \textit{IF}-coloring of $G$. \qedhere
\end{itemize}
\end{proof}

\begin{lemma}\label{FACE}
Let $[x_{1}x_{6}x_{7}\dots x_{m}]$ be an internal face with $7 \leq m \leq 9$, and let $[x_{1}x_{2}x_{3}x_{4}x_{5}x_{6}]$ be an internal $6$-face. Let $x$ be the third neighbor of $x_{m}$. If $d(x_{i}) = 3$ for each $i \neq 2$, and $x$ is an internal vertex, then $d(x_{2}) \geq 5$.
\end{lemma}
\begin{proof}
Suppose for contradiction that $d(x_{2}) \leq 4$. By \autoref{Y}, we must have $d(x_{2}) = 4$. Let $x_{1}, x_{3}, x'$ and $x''$ be the four neighbors of $x_{2}$ in a cyclic order. Let $G_{0} = G - \{x_{1}, x_{2}, \dots, x_{m}\}$, and let $G_{1}$ be obtained from $G_{0}$ by identifying $x$ and $x'$. Using arguments similar to \autoref{Y}, we can prove that the distance between $x$ and $x'$ in $G_{0}$ is at least $6$. Consequently, $C_{0}$ bounds the outer face of $G_{1}$, and has no chords in $G_{1}$.  Since $x$ is an internal vertex and has distance $1$ from the $5^{-}$-face $[x_{1}x_{6}x_{7}\dots x_{m}]$, we have that $C_{0}$ is a good cycle in $G_{1}$. 

By the minimality of $G$, we can extend the precoloring $\phi_{0}$ to a super \textit{IF}-coloring of $G_{1}$, and thereby to a super \textit{IF}-coloring $\phi$ of $G_{0}$, with $x$ and $x'$ having the same color. 

Suppose that $\phi(x) = \phi(x') = I$. We color $x_{m}$ and $x_{2}$ with $F$, and nicely color $x_{3}, x_{4}, x_{5}$, $x_{m-1}, \dots, x_{6}$. The resulting coloring has no barriers. Finally, we color $x_{1}$ with a color different from $x_{6}$. If $x_{6}$ is colored with $F$, then $x_{1}$ is colored with $I$, and the final coloring has no barriers. If $x_{6}$ is colored with $I$, then $x_{1}$ is colored with $F$, which implies that $x_{m}x_{1}x_{2}$ is not contained in a barrier by \autoref{prop}. Therefore, the resulting coloring is a super \textit{IF}-coloring of $G$, a contradiction. 

Suppose that $\phi(x) = \phi(x') = F$. We color $x_{1}$ with $F$, and $x_{2}$ with a color different from $x''$, and then nicely color $x_{3}, x_{4}, \dots, x_{m-1}$. Finally, we color $x_{m}$ with a color different from $x_{m-1}$. It is easy to check that the path $xx_{m}x_{1}$ is not on a barrier. Therefore, the resulting coloring is a super \textit{IF}-coloring of $G$, a contradiction. 
\end{proof}

For convenience, let $F_{0} \coloneqq F(G)\setminus \{D\}$, and $F_{1} \coloneqq \{f \mid f\in F_{0} \textrm{ and } V(f) \cap V(D) \neq\emptyset\}$. We introduce some definitions to facilitate the proof. A \emph{special $3$-face} is an internal $3$-face that has at most one $4^{+}$-vertex. A \emph{special $6$-face} is an internal $6$-face that is adjacent to a special $3$-face. A \emph{bad face} is a $6^{+}$-face in $F_{1}$ that shares an edge with a $5^{-}$-face in $E(C_{0}, G - C_{0})$. A \emph{pendent face} of a vertex $v$ is a face $f$ that does not contain $v$, but has a neighbor of $v$. Conversely, $v$ is a \emph{pendent vertex} of $f$. Let $g$ be an internal $(3, 3, 3, 3, 3, 3)$-face that is adjacent to an internal $3$-face $g'$. We say that the vertex on $g'$ but not on $g$ is a \emph{roof} of $g$, and $g$ is a \emph{base} of $v$. 

We now describe the discharging part. First, we assign an initial charge $\mu(v) = 2d(v) - 6$ to each vertex $v \in V(G)$, $\mu(f) = d(f) - 6$ to each face $f \in F(G)\setminus \{D\}$, and $\mu(D) = d(f) + 6$ to the outer face $D$. By Euler's formula and the handshaking lemma, we have $\sum_{x \in V(G) \cup F(G)} \mu(x) = 0$. We then design some discharging rules to redistribute the charges among vertices and faces so that each element $x$ in $V(G) \cup F(G)$ has a nonnegative final charge $\mu^{*}(x)$, and $\mu^{*}(D) > 0$. This leads a contradiction that \[0 = \sum_{x \in V(G) \cup F(G)} \mu(x) = \sum_{x \in V(G) \cup F(G)} \mu^{*}(x) > 0.\] 

The following are the discharging rules. 
\begin{enumerate}[label = \textbf{R\arabic*.}, ref = R\arabic*]
\item\label{R1} Let $v$ be an internal $4^{+}$-vertex. 
\begin{enumerate}
\item\label{R1a} If $d(v) = 4$ and $v$ is incident with a $3$-face, then $v$ sends $\frac{3}{2}$ to its incident internal $3$-face, and $\frac{1}{2}$ to the base (if it exists). 
\item\label{R1b} If $d(v) \geq 5$ and $v$ is incident with a $3$-face, then $v$ sends $\frac{3}{2}$ to its incident internal $3$-face, and $\frac{1}{2}$ to each incident $6$-face and the base.  
\item\label{R1c} If $v$ is incident with a $4$-face, then $v$ sends $2$ to the $4$-face.
\item\label{R1d} If $v$ is incident with a $5$-face, then $v$ sends $1$ to the $5$-face. 
\item\label{R1e} If $d(v) = 4$ and $v$ has a pendent internal $5^{-}$-face $g$ in which all vertices are $3$-vertices, then $v$ sends $1$ to $g$ and $\frac{1}{2}$ to each incident $6$-face that is not adjacent to $g$.
\item\label{R1f} If $d(v) \geq 5$ and $v$ has a pendent internal $5^{-}$-face $g$ in which all vertices are $3$-vertices, then $v$ sends $2$ to $g$ and $\frac{1}{2}$ to each incident $6$-face that is not adjacent to $g$.
\item\label{R1g} If $v$ has no pendent internal $5^{-}$-face in which all vertices are $3$-vertices, and is not incident with a $5^{-}$-face, then it sends $\frac{1}{2}$ to each incident $6$-face.
\end{enumerate}

\item\label{R2} Every $6^{+}$-face in $F_{0}$ that is not an internal $6$-face sends $1$ to each adjacent $5^{-}$-face other than $D$ and the remaining charge to the outer face.  
\item\label{R3} Each internal $(3, 3, 4^{+})$-face receives $\frac{1}{2}$ from each adjacent internal $6$-face. Each internal $(3, 3, 3)$-face $f$ receives $\frac{1}{2}$ from each adjacent internal $6$-face which contains a non-pendent $4^{+}$-vertex of $f$.

\item\label{R4} The outer face $D$ receives $\mu(v)$ from each boundary vertex $v$, and sends $1$ to each $3$-face in $F_{1}$, and each $6$-face in $F_{1}$ which is adjacent to a $5^{-}$-face. 
\end{enumerate}

By \ref{R4}, every boundary vertex sends its initial charge to the outer face $D$, resulting in a final charge of zero. Therefore, our consideration can be confined to internal vertices. 

Let $v$ be an arbitrary internal vertex. By Lemma~\ref{mindeg}, we have that $d(v)\geq 3$. If $d(v) = 3$, then the vertex is not involved in the discharging process, so $\mu^{*}(v) = \mu(v) = 0$. 

Let $d(v) = 4$ and $f_{1}, f_{2}, f_{3}, f_{4}$ be the four incident faces of $v$ in a cyclic order. 
\begin{enumerate*}[label = (\roman*)]
\item Assume $v$ is incident with a $3$-face $f_{1}$. Since $d^{*} \geq 3$, we have that $v$ is incident with precisely one $5^{-}$-face and has no pendent $5^{-}$-faces. Moreover, $f_{3}$ cannot be an internal $6$-face adjacent to a $5^{-}$-face. Note that every internal $6$-face is adjacent to at most one $5^{-}$-face. By \ref{R1}, $v$ sends at most $\frac{3}{2}$ to $f_{1}$, $\frac{1}{2}$ to the base if it exists, and no charge to $f_{2}, f_{3}$ or $f_{4}$. Thus, $\mu^{*}(v) \geq 2 - \frac{3}{2} - \frac{1}{2} = 0$. 
\item Assume $v$ is incident with a $4$-, or $5$-face $f_{1}$. In this case, it is incident with precisely one $5^{-}$-face and has no pendent $5^{-}$-face. When $f_{1}$ is a $4$-face, $\mu^{*}(v) \geq 2 - 2 = 0$; while if $f_{1}$ is a $5$-face, $\mu^{*}(v) \geq 2 - 1 = 1$.
\item Next, assume $v$ is not incident with a $5^{-}$-face but has a pendent $5^{-}$-face $g$ in which all vertices are $3$-vertices. Since $d^{*} \geq 3$, $v$ has only one pendent $5^{-}$-face. In this case, $v$ sends $1$ to $g$ and $\frac{1}{2}$ to at most two incident faces, so $\mu^{*}(v) \geq 2 - 1 - \frac{1}{2} \times 2 = 0$. 
\item Finally, consider the case that $v$ has no pendent internal $5^{-}$-face in which all vertices are $3$-vertices, and is not incident with a $5^{-}$-face. By \ref{R1g}, $v$ sends at most $\frac{1}{2}$ to each incident face. Then $\mu^{*}(v) \geq 2 - \frac{1}{2} \times 4 = 0$. 
\end{enumerate*}

Let $d(v) = k \geq 5$. If $v$ is incident with a $3$-face $f_{1}$, then it sends at most $\frac{3}{2}$ to $f_{1}$, at most $\frac{1}{2}$ to each other incident face, and $\frac{1}{2}$ to the base if it exists, which implies that $\mu^{*}(v) \geq 2k - 6 - \frac{3}{2} - \frac{1}{2} \times k = \frac{3(k - 5)}{2}\geq 0$ by \ref{R1b}. If $v$ is incident with a $4$- or $5$-face $f_{1}$, then it sends at most $2$ to $f_{1}$, implying $\mu^{*}(v) \geq 2k - 6 - 2 > 0$. If $v$ is not incident with a $5^{-}$-face but has a pendent face $g$ in which all vertices have degree $3$, then $v$ sends $2$ to $g$, and $\frac{1}{2}$ to each incident special $6$-faces that is not adjacent to $g$, so $\mu^{*}(v) \geq 2k - 6 - 2 - \frac{1}{2} \times (k - 2) = \frac{3k-14}{2} > 0$. If $v$ has no pendent $5^{-}$-face in which all vertices have degree $3$, and is not incident with a $5^{-}$-face, then $v$ sends at most $\frac{1}{2}$ to each incident face, so $\mu^{*}(v) \geq 2k - 6 - \frac{1}{2} \times k > 0$. 

For a $7^{+}$-face $f$ in $F_{0}$, the number of incident $5^{-}$-faces is at most $\lfloor\frac{d(f)}{4}\rfloor$. Therefore, by \ref{R2}, we have $\mu^{*}(f) \geq d(f) - 6 - 1 \times \lfloor\frac{d(f)}{4}\rfloor \geq 0$ .  

Consider a $6$-face $f \neq D$. It is known that $f$ is adjacent to at most one $5^{-}$-face. If $f$ is a $6$-face in $F_{1}$ that is adjacent to a $5^{-}$-face other than $D$, then it receives $1$ from the outer face $D$ by \ref{R4}, and sends at most $1$ to the adjacent $5^{-}$-face by \ref{R2}, thus $\mu^{*}(f) \geq 0 + 1 - 1 = 0$. If $f$ is an internal non-special $6$-face or a non-internal $6$-face which is not adjacent to a $5^{-}$-face, then it does not send out charge, so $\mu^{*}(f) \geq \mu(f) = 0$. So we may assume that $f$ is a special $6$-face. Now, assume $f = [v_{1}v_{2}\dots v_{6}]$ and it is adjacent to a special $3$-face $[v_{1}v_{6}v_{7}]$ with $d(v_{1}) \leq d(v_{6})$. Since $[v_{1}v_{6}v_{7}]$ is a special $3$-face, we have $d(v_{1}) = 3$. 

\begin{itemize}
\item $d(v_{6}) \geq 5$. Then $d(v_{7}) = 3$, $f$ receives $\frac{1}{2}$ from $v_{6}$, and sends $\frac{1}{2}$ to the adjacent special $3$-face $[v_{1}v_{6}v_{7}]$. Thus, $\mu^{*}(f) \geq 0 + \frac{1}{2} - \frac{1}{2} = 0$. 

\item $d(v_{6}) = 4$. Then $d(v_{7}) = 3$. By \autoref{34Chord}, $f$ has a $4^{+}$-vertex other than $v_{6}$. If $v_{2}$ or $v_{5}$ is a $4^{+}$-vertex, then it sends $\frac{1}{2}$ to $f$ by \ref{R1g}, so $\mu^{*}(f) \geq 0 + \frac{1}{2} - \frac{1}{2} = 0$. If one of $v_{3}$ and $v_{4}$ is a $4^{+}$-vertex, then it sends $\frac{1}{2}$ to $f$ by \ref{R1e}, \ref{R1f} and \ref{R1g}, so $\mu^{*}(f) \geq 0 + \frac{1}{2} - \frac{1}{2} = 0$. 

\item $d(v_{6}) = 3$ and $d(v_{7}) \geq 4$.  If $f$ is a $(3, 3, 3, 3, 3, 3)$-face, then $f$ receives $\frac{1}{2}$ from its roof $v_{7}$ and sends $\frac{1}{2}$ to $[v_{1}v_{6}v_{7}]$, so $\mu^{*}(f) = 0 + \frac{1}{2} - \frac{1}{2} = 0$. If $f$ is incident with a $4^{+}$-vertex, then $f$ receives at least $\frac{1}{2}$ from each incident $4^{+}$-vertex and sends $\frac{1}{2}$ to $[v_{1}v_{6}v_{7}]$, so $\mu^{*}(f) \geq 0 + \frac{1}{2} - \frac{1}{2} = 0$. 

\item $d(v_{6}) = d(v_{7}) = 3$. Then at least one of $v_{2}$ and $v_{5}$ is a $4^{+}$-vertex by \autoref{Y}. If one of $v_{3}$ and $v_{4}$ is a $4^{+}$-vertex, then each $4^{+}$-vertex in $\{v_{3}, v_{4}\}$ sends $\frac{1}{2}$ to $f$ by \ref{R1g}, so $\mu^{*}(f) \geq 0 + \frac{1}{2} - \frac{1}{2} = 0$. If neither $v_{3}$ nor $v_{4}$ is a $4^{+}$-vertex, then $f$ does not send charge to $[v_{1}v_{6}v_{7}]$, so $\mu^{*}(f) \geq \mu(f) = 0$. 
\end{itemize}

Let $f = [x_{1}x_{2}\dots]$ be a face with $d(f) \in \{4, 5\}$. If $f \in F_{1}$, then $f$ receives $1$ from each adjacent bad face, implying $\mu^{*}(f) \geq \mu(f) + 1 \times 2 \geq 0$. Assume $f$ is an internal face. If $f$ contains a $4^{+}$-vertex and $d(f) = 4$, then $\mu^{*}(f) \geq \mu(f) + 2 = 0$; if $f$ contains a $4^{+}$-vertex and $d(f) = 5$, then $\mu^{*}(f) \geq \mu(f) + 1 = 0$. Then every vertex on $f$ is a $3$-vertex. If none of the adjacent face is an internal $6$-face, then $\mu^{*}(f) \geq \mu(f) + 1 \times d(f) > 0$. Assume $[x_{1}x_{2}y_{3}y_{4}y_{5}y_{6}]$ is an internal $6$-face. By \autoref{Y}, either $y_{3}$ or $y_{6}$ is a $4^{+}$-vertex. Then we may assume that every adjacent face of $f$ is an internal $6$-face; otherwise $f$ receives at least $1$ from $\{y_{3}, y_{6}\}$ by \ref{R1e} and \ref{R1f}, and at least $1$ from adjacent faces by \ref{R2}, implying $\mu^{*}(f) \geq \mu(f) + 1 + 1 \geq 0$. Once again, \autoref{Y} implies that $f$ has at least two pendent $4^{+}$-vertices. This implies that $f$ receives at least $1$ from each pendent $4^{+}$-vertex by \ref{R1e} and \ref{R1f}, and $\mu^{*}(f) \geq \mu(f) + 1 \times 2 \geq 0$.

Next, consider a $3$-face $f = [x'y'z']$, and let $f_{1}, f_{2}, f_{3}$ be the three faces adjacent to $y'z', x'z', x'y'$, respectively. If $f$ has a common vertex with $D$, then it receives $1$ from $D$ by \ref{R4}, and $1$ from each adjacent $7^{+}$-face and bad $6$-face by \ref{R2}, thus $\mu^{*}(f) \geq \mu(f) + 1 + 1 \times 2 = 0$. So we may assume that $f$ is an internal face. If $f$ is incident with at least two $4^{+}$-vertices, then $f$ receives $\frac{3}{2}$ from each incident $4^{+}$-vertex by \ref{R1a} and \ref{R1b}, thus $\mu^{*}(f) \geq \mu(f) + \frac{3}{2} \times 2 = 0$. If $f$ is a $(3, 3, 4^{+})$-face, then $f$ receives $\frac{3}{2}$ from the incident $4^{+}$-vertex, and at least $\frac{1}{2}$ from each adjacent $6^{+}$-face \ref{R2} and \ref{R3},  thus $\mu^{*}(f) \geq \mu(f) + \frac{3}{2} + \frac{1}{2} \times 3 = 0$.

Now, let's assume that $f$ is an internal $(3, 3, 3)$-face. Denote $x, y, z$ as the third neighbors of $x', y', z'$, respectively. Since $d^{*} \geq 3$, we have that $x, y$ and $z$ are distinct. If none of the adjacent faces is an internal $6$-face, then $f$ receives $1$ from each adjacent face, leading to $\mu^{*}(f) = \mu(f) + 1 \times 3 = 0$. Assume that $f$ is adjacent to an internal $6$-face $f_{1} = [yy'z'zpq]$. By \autoref{Y}, either $y$ or $z$, say $y$, is a $4^{+}$-vertex. If neither $f_{2}$ nor $f_{3}$ is an internal $6$-face, then $\mu^{*}(f) \geq \mu(f) + 1 + 1 \times 2 = 0$. Hence, we can assume that $f_{2}$ or $f_{3}$ is an internal $6$-face. This requires $x$ to be an internal vertex. By \autoref{Y}, one of $x$ and $z$ is a $4^{+}$-vertex. Knowing that at least two of $x, y$ and $z$ are internal $4^{+}$-vertices, and if any one of them is a $5^{+}$-vertex, then $\mu^{*}(f) \geq \mu(f) + 1 + 2 = 0$. We can then safely assume that $x, y$ and $z$ are all $4^{-}$-vertices, with $d(y) = 4$ specifically. If $d(x) = d(z) = 4$, then $f$ receives $1$ from each of $x, y$ and $z$, resulting in $\mu^{*}(f) \geq \mu(f) + 1 \times 3 = 0$. So we need to consider the case that two of $x, y$ and $z$ are $4$-vertices and the other one is a $3$-vertex. If one of $f_{2}$ and $f_{3}$ is not an internal $6$-face, then $f$ receives $1$ from $f_{2}$ or $f_{3}$, which implies that $\mu^{*}(f) \geq \mu(f) + 1 \times 2 + 1 = 0$. Hence, we assume that $f_{1}, f_{2}$ and $f_{3}$ are all internal $6$-faces. By symmetry, assume $d(z) = 3$ and $d(x) = 4$. By \autoref{FACE}, each of $f_{1}$ and $f_{2}$ is incident with a $4^{+}$-vertex such that the $4^{+}$-vertex is not a pendent vertex of $f$. By \ref{R1e} and \ref{R3}, $f$ receives $\frac{1}{2}$ from each of $f_{1}$ and $f_{2}$. Hence, $\mu^{*}(f) \geq \mu(f) + 1 \times 2 + \frac{1}{2} \times 2 = 0$. 


Let $\tau_{3}$ denote the number of $3$-faces in $F_{1}$, $\tau_{6}$ denote the number of $6$-faces in $F_{1}$ that are adjacent to a $5^{-}$-face in $F_{0}$, $e'$ denote the number of edges in $E(C_{0}, G - C_{0})$ that are not on any 3-face, and $p$ denote the total charge that $D$ received from other faces by \ref{R2}. 
\[
\begin{aligned}
\mu^{*}(D) & = d(D) + 6 + \sum_{v \in C_{0}}(2d(v) - 6) - \tau_{3} - \tau_{6} + p\\
&= d(D) + 6 + \sum_{v \in C_{0}}(2d(v) - 4)  - 2d(D) - \tau_{3} - \tau_{6} + p\\
&= 6 - d(D) + 2|E(C_{0}, G - C_{0})| - \tau_{3} - \tau_{6} + p\\
&= 6 - d(D) + 2(2\tau_{3} + e')  - \tau_{3} - \tau_{6} + p\\
& = 6 - d(D) + 3\tau_{3} + 2e' - \tau_{6} + p.
\end{aligned}
\]

Our goal is to prove that $\mu^{*}(D) > 0$. Suppose otherwise that $\mu^{*}(D) \leq 0$. Then 
\begin{equation}
6 - d(D) + 3\tau_{3} + 2e' - \tau_{6} + p \leq \mu^{*}(D) \leq 0.
\end{equation} 
Since $d(D) \leq 9$, we have 
\begin{equation}\label{EQ2}
6 + 3\tau_{3} + 2e' - \tau_{6} + p \leq d(D) \leq 9.
\end{equation}
Observing that every face counted in $\tau_{6}$ has at least one edge which is counted in $e'$, and every edge counted in $e'$ is contained in at most two faces counted in $\tau_{6}$. Therefore, 
\begin{equation}\label{EQ3}
2e' \geq \tau_{6},
\end{equation}
with equality only if every face counted in $\tau_{6}$ contains precisely one edge counted in $e'$, and precisely one edge on a $3$-face in $F_{1}$. By \eqref{EQ2} and \eqref{EQ3}, we have $\tau_{3} \leq 1$. 

First, assume $\tau_{3} = 1$. It follows from \eqref{EQ2} and \eqref{EQ3} that $d(D) = 9$ and $2e' - \tau_{6} = p = 0$. Since $\tau_{3} = 1$, we have $\tau_{6} \leq 2\tau_{3} = 2$. If $2e' = \tau_{6} > 0$, then $2e' = \tau_{6} = 2$, which implies that $C_{0}$ is the bad $9$-cycle with a claw as depicted in \autoref{FB1}, a contradiction. If $2e' = \tau_{6} = 0$, then $E(C_{0}, G - C_{0})$ consists of two edges which are in a $3$-face, thus $G$ has a face $f$ with at least seven consecutive $2$-vertices, which implies $p > 0$, a contradiction.

Next, assume $\tau_{3} = 0$. Then 
\begin{equation}
e' = |E(C_{0}, G - C_{0})| \geq |F_{1}| \geq \tau_{6}.
\end{equation}
Utilizing \eqref{EQ2}, we derive the inequality:
\begin{equation}\label{EQ5}
e' + (e' - \tau_{6}) + p \leq d(D) - 6 \leq 3.
\end{equation}
Hence, $e' = |E(C_{0}, G - C_{0})| \leq 3$. We consider three cases according to the value of $e'$. 

\begin{itemize}
\item If $e' = 1$, then the edge in $E(C_{0}, G - C_{0})$ is a cut edge, and the incident face contains a walk of length $d(D) + 2$ such that no charge is sent through it, implying $\tau_{6} = 0$. By \eqref{EQ5}, $d(D) \geq 8$. However, this leads to $p \geq d(f) - 6 - \lceil \frac{d(f) - (d(D) + 2)}{4}\rceil \geq 2$, contradicting \eqref{EQ5}. 
\item Assume that $e' = 2$. It follows from \eqref{EQ5} that $d(D) \geq 8$. If $e' - \tau_{6} = 0$, then $F_{1}$ consists of two $6$-faces, thus the two edges in $E(C_{0}, G - C_{0})$ have a common endpoint in $\int(C_{0})$ and $|\int(C_{0})| = 1$, which implies that the internal vertex has degree $2$, contradicting \autoref{mindeg}. Therefore, $e' - \tau_{6} \neq 0$. Consequently, $d(D) = 9$ and $p = 0$. However, there exists a $7^{+}$-face $f$ with at least four consecutive 2-vertices, thus $p \geq d(f) - 6 - \lceil \frac{d(f) - 7}{4} \rceil \geq 1$, which contradicts the conclusion $p = 0$. 
\item If $e' = 3$,  following \eqref{EQ5}, we deduce that $d(D) = 9$, $\tau_{6} = e' = 3$ and $p = 0$. As a result, $C_{0}$ is the bad $9$-cycle with a triclaw as depicted in \autoref{FB2}, a contradiction. 
\end{itemize}


\begin{thebibliography}{10}

\bibitem{MR3686937}
A.~Bernshteyn, A.~V. Kostochka and S.~P. Pron, On {DP}-coloring of graphs and
  multigraphs, Sib. Math. J. 58~(1) (2017) 28--36.

\bibitem{MR1918259}
O.~V. Borodin and A.~N. Glebov, On the partition of a planar graph of girth 5
  into an empty and an acyclic subgraph, Diskretn. Anal. Issled. Oper. Ser. 1
  8~(4) (2001) 34--53.

\bibitem{MR3759461}
M.~Chen, W.~Yu and W.~Wang, On the vertex partitions of sparse graphs into an
  independent vertex set and a forest with bounded maximum degree, Appl. Math.
  Comput. 326 (2018) 117--123.

\bibitem{MR4134028}
D.~W. Cranston and M.~P. Yancey, Sparse graphs are near-bipartite, SIAM J.
  Discrete Math. 34~(3) (2020) 1725--1768.

\bibitem{MR3785024}
F.~Dross, M.~Montassier and A.~Pinlou, Partitioning sparse graphs into an
  independent set and a forest of bounded degree, Electron. J. Combin. 25~(1)
  (2018) P1.45.

\bibitem{MR3758240}
Z.~Dvo\v{r}\'{a}k and L.~Postle, Correspondence coloring and its application to
  list-coloring planar graphs without cycles of lengths 4 to 8, J. Combin.
  Theory Ser. B 129 (2018) 38--54.

\bibitem{arXiv:2303.04648}
Y.~Kang, H.~Lu and L.~Jin, {$(I,F)$}-partition of planar graphs without cycles
  of length 4, 6, and 9, Discuss. Math. Graph Theory  (2023)
  \url{https://doi.org/10.7151/dmgt.2523}.

\bibitem{MR2518200}
K.-i. Kawarabayashi and C.~Thomassen, Decomposing a planar graph of girth 5
  into an independent set and a forest, J. Combin. Theory Ser. B 99~(4) (2009)
  674--684.

\bibitem{MR4115511}
R.~Liu and G.~Yu, Planar graphs without short even cycles are near-bipartite,
  Discrete Appl. Math. 284 (2020) 626--630.

\bibitem{MR4422988}
F.~Lu, M.~Rao, Q.~Wang and T.~Wang, Planar graphs without normally adjacent
  short cycles, Discrete Math. 345~(10) (2022) 112986.

\bibitem{MR3954054}
Y.~Yin and G.~Yu, Planar graphs without cycles of lengths 4 and 5 and close
  triangles are {DP}-3-colorable, Discrete Math. 342~(8) (2019) 2333--2341.

\bibitem{Zhao_2020}
Y.~Zhao and L.~Miao, Every planar graph with the distance of $5^{-}$-cycles at
  least 3 from each other is {DP}-3-colorable, Mathematics 8~(11) (2020) 1920.

\end{thebibliography}
\end{document}